\documentclass{amsart}

\usepackage{amsfonts}
\usepackage{amssymb, amsmath, mathrsfs}
\usepackage{hhline}

 \newtheorem{thm}{Theorem}[section]
 \newtheorem{cor}[thm]{Corollary}
 \newtheorem{lem}[thm]{Lemma}
 \newtheorem{prop}[thm]{Proposition}
 \theoremstyle{definition}
 
 \theoremstyle{remark}
 \newtheorem{rem}[thm]{Remark}
 
 \numberwithin{equation}{section}

\begin{document}

\title[$\beta-$Sturm Liouville problem]
 {A $\beta-$Sturm Liouville problem associated with \\
 the general quantum operator \footnote{The final version of this preprint will be published in Journal of Difference Equations
and Applications}}

\author[Cardoso]{J.L. Cardoso}

\address{
Dep. de Matem\'atica da Escola de Ci\^encias e Tecnologia \\
Universidade de Tr\'as-os-Montes e Alto Douro (UTAD) \\
Quinta de Prados\\
5001-801 Vila Real\\
Portugal \\
ORCID ID: 0000-0002-7418-3634}
\email{jluis@utad.pt}

\thanks{This research was partially supported by Portuguese Funds through FCT
(Funda\c c\~ao para a Ci\^encia e a Tecnologia) within the Projects
UIDB/00013/2020 and UIDP/00013/2020.}



\date{May 18, 2021}


\begin{abstract}
Let $\,I\subseteq\mathbb{R}\,$ be an interval and $\,\beta:\,I\rightarrow\,I\,$
a strictly increasing and continuous function with a unique fixed point
$\,s_0\in I\,$ that satisfies $\,(s_0-t)(\beta(t)-t)\geq 0\,$ for all $\,t\in I$,
where the equality holds only when $\,t=s_0$.

\noindent The general quantum operator defined by Hamza et al.,
$\,D_{\beta}[f](t):=\displaystyle\frac{f\big(\beta(t)\big)-f(t)}{\beta(t)-t}\,$
if $\,t\neq s_0\,$ and $\,D_{\beta}[f](s_0):=f^{\prime}(s_0)\,$ if
$\,t=s_0,$ generalizes the Jackson $\,q$-operator $\,D_{q}\,$
and also the Hahn (quantum derivative) operator, $\,D_{q,\omega}$.

Regarding a $\beta-$Sturm Liouville eigenvalue problem
associated with the above operator $\,D_{\beta}\,$, we construct
the $\beta-$Lagrange's identity, show that it is self-adjoint
in $\,\mathscr{L}_{\beta}^2([a,b]),$ and exhibit some properties for the
corresponding eigenvalues and eigenfunctions.
\end{abstract}

\maketitle

\section{Introduction}
The $\beta-$operator $\,D_{\beta}\,$ described in the abstract was introduced in
\cite{HSSA:2015} together with the corresponding general quantum difference calculus.
It generalizes the $(q,\omega)-$derivative operator (the Hahn's quantum operator)
\begin{equation}\label{Hahn-operator}
D_{q,\omega}[f](x):=\frac{f\big(qx+\omega\big)-f(x)}{(q-1)x+\omega}\,,
\end{equation}
which, in turn, generalizes both the Jackson $q-$derivative
\begin{equation*}
D_{q}[f](x):=\frac{f(qx)-f(x)}{(q-1)x}\,,
\end{equation*}
and the (forward difference) $\omega-$derivative
$$
\triangle_{\omega}[f](x):=\frac{f(x+\omega)-f(x)}{\omega}\,,
$$
where $\,0<q<1$ and $\,\omega\geq 0\,$ are fixed parameters.

Of particular relevance are the corresponding inverse operators, which enable one to
define, the following integrals, respectively: the Jackson-Thomae-N\"orlund
$\,(q,\omega)-$integral
\begin{equation*}
\int_a^b f\,{\rm d}_{q,\omega}:=
\int_{\omega_0}^b f\,{\rm d}_{q,\omega}-\int_{\omega_0}^a f\,{\rm d}_{q,\omega}\,,
\end{equation*}
the Jackson $\,q-$integral
\begin{equation*}
\int_{a}^{b}f\,{\rm d}_q:=(1-q)\sum_{k=0}^{+\infty}\big[bf(bq^k)-af(aq^k)\big]q^k
\end{equation*}
and the N\"orlund integral
$$
\int_{a}^{b}\!f\,\Delta_{\omega}:=\omega\sum_{k=0}^{+\infty}\big[f(b+k\omega)-f(a+k\omega)\big]\,.
$$
For more details over the $\,q-$integrals see, for example, \cite{AHA:2012}.

Those difference operators together with its inverse operators are very important in
mathematics investigation and in applications, with a large number of publications
and a variety of topics including, but not limited to,
the quantum calculus \cite{AM:2012},
the quantum variational calculus
\cite{MT:2010, CMT:2012, AMT:2012, MT:2014, M:2016},
$\,q-$difference equations properties \cite{AAIM:2007, A:2009, JCF:2017, AHM:2018},
Sturm-Liouville problems \cite{E:1992, AM:2005, ABI:2007, AHM:2012, M:2016, AHM:2018},
Paley-Wiener results \cite{FD:2004, A:2007, DBF:2009},
Sampling theory \cite{A:2003, IZ:2003, A:2004, DBF:2009, AHM:2012, AH:2018, H:2019},
$\,q-$exponential, trigonometric, hyperbolic and other important families of functions
associated with Fourier expansions and corresponding properties connected and derived
from its orthogonality feature \cite{KS:1992, BC:2001, ABC:2003, S:2003, A:2006, C:2006,
P:2007, C:2011, CP:2015, C:2016, C:2018, AA-NC:2019, C:2020}.
These and many other subjects has attracted many researchers.

In 2015, Hamza et al. \cite{HSSA:2015} introduced a general quantum difference operator,
the $\beta-$derivative, generalizing the Hahn's quantum operator
(for certain functions $\,\beta$), and its inverse operator, the $\beta-$integral.
Also in 2015 \cite{HS:2015}, $\,\beta-$H\"older, $\,\beta-$Minkowski, $\,\beta-$Gronwall,
$\,\beta-$Bernoulli and $\,\beta-$Lyapunov inequalities were exhibited. In 2016, it was
proved the existence and uniqueness of solutions of general quantum difference equations
\cite{HS:2016}.
Later, in \cite{FSZ:2017}, some new results on homogeneous second order linear general
quantum difference equations were presented and, in \cite{HSS:2017}, the exponential,
trigonometric and hyperbolic functions were introduced. The theory of $n$th-order linear
general quantum difference equations was developed in \cite{FSZ:2018} while the general
quantum variational calculus was build up in \cite{CM:2018}. In \cite{C:2020}, properties
of the $\beta$-Lebesgue spaces were produced and, recently, in \cite{SFZ:2020},
a general quantum Laplace transform was displayed and studied.

\noindent All of these publications generalizes previous results and are directly
related with the above mentioned general quantum difference operator.

The aim of this work is to obtain a self-adjoint formulation of a
$\beta$-Sturm-Liouville problem and to prove properties for the
corresponding eigenvalues and eigenfunctions. Its construction follows
ideas from \cite{AHM:2018} and from other previous publications.

In section \ref{Sect-2} we recall the definition of the $\beta$-difference
operator and its inverse operator, the $\beta$-integral, together with some of
its properties. Section \ref{Sect-3} is devoted to the $\beta$-Sturm-Liouville
problem to be considered.

The outcome of this work can be found in section \ref{Sect-3}.
We believe that Lemmas 3.1 and 3.2, Corollaries 3.3, 3.5, 3.8, 3.9 and 3.10,
as well as Theorems \ref{beta-Lagrange-Identity} and \ref{beta-Lemma3.5} are original.
Subsection \ref{particular-case} is also new.

\section{The $\beta-$difference operator and the $\beta-$integral}\label{Sect-2}

\subsection{The $\beta-$difference operator}\label{Subsect-2.1}
In the following, $\,I\subseteq\mathbb{R}\,$ will denote an interval and
$\,\beta:\,I\rightarrow\,I\,$ a strictly increasing and continuous function with
a unique fixed point $\,s_0\in I\,$ satisfying
\begin{equation}\label{condition1}
(t-s_0)(\beta(t)-t)\leq 0
\end{equation}
for all $\,t\in I$, where the equality holds only when $\,t=s_0\,$.

For each function $\,f:\,I\rightarrow\,\mathbb{K}\,$ where
$\,\mathbb{K}\,$ is either $\,\mathbb{R}\,$ or $\,\mathbb{C}\,$
\footnote{In fact, $\mathbb{K}=\mathbb{X}$ can represent any Banach space
\cite[p.2]{HSSA:2015}},
Hamza et al. \cite{HSSA:2015} defined the general quantum difference operator
\begin{equation}\label{beta-derivative}
D_{\beta}[f](t):=\left\{
\begin{array}{ccl}
\displaystyle\frac{f\big(\beta(t)\big)-f(t)}{\beta(t)-t} & \mbox{\rm if} &
t\neq s_0\,, \\ [1.2em] f^{\prime}(s_0) & \mbox{\rm if} &
t=s_0\; .
\end{array}\right.,
\end{equation}
provided that $\,f^{\prime}(s_0)\,$ exists. $\,D_{\beta}[f](t)\,$ is
called the $\beta-$derivative of $\,f\,$ at $\,t\in I$.
If $\,f^{\prime}(s_0)\,$ exists then $\,f\,$ is said to be $\beta-$differentiable on $\,I\,$.

It is obvious that when $\,\beta(t)=qt+\omega$ one obtains the Hahn operator
(\ref{Hahn-operator}), being the fixed point given by $\,s_0=\frac{\omega}{1-q}$.

We remark that it is possible to replace the above condition (\ref{condition1}) by
$\,(t-s_0)(\beta(t)-t)\geq 0\,$ for $\,t\in I$.

An introduction to the $\,\beta-$calculus related with this general quantum
difference operator can be found in \cite{HSSA:2015}.

\subsection{The $\beta-$integral}

Defining
$\,\beta^k(t):=(\underbrace{\beta\circ\beta\circ\ldots\circ\beta}_{k\,\mbox{times}})(t)\,$
for $\,t\in I\,$ and $\,k\in \mathbb{N}_0=\mathbb{N}\cup\{0\}$, with $\,\beta^0(t):=t,$
we can consider the $\,\beta-$interval with extreme points $\,a\,$ and $\,b,$
\begin{equation*}
[a,b]_{\beta}:=\big\{\,\beta^n(x)\; |\; (x,n)\in \{a,b\}\times\mathbb{N}_0\,\big\}.
\end{equation*}
Of course that
$$
a,b\in I\quad\Rightarrow\quad [a,b]_{\beta}\subset I\,,
$$
whenever $\,a\,$ and $\,b$ are real numbers.

\vspace{0.5em}
The following proposition can be found in \cite[Lemma 2.1, page 3]{HSSA:2015}.

\noindent
\textbf{Proposition A}.
The sequence of functions $\,\left\{\beta^k(t)\right\}_{k\in\mathbb{N}_0}\,$ converges
uniformly to the constant function $\,\hat{\beta}(t):=s_0\,$ on every compact interval
$\,J\subset\,I\,$ containing $\,s_0$.

\vspace{0.5em}
The quantum difference inverse operator, the $\beta-$integral, with $\,a,b\in I,$
is defined by
\begin{equation}\label{beta-integral-a-b}
\int_{a}^{b}f\,{\rm d}_{\beta}:=\int_{s_0}^{b}f\,{\rm d}_{\beta}-\int_{s_0}^{a}f\,{\rm d}_{\beta}
\end{equation}
where
\begin{equation}\label{beta-integral-s0-x}
\int_{s_0}^{x}f\,{\rm d}_{\beta}:=
\sum_{k=0}^{+\infty}\Big(\beta^k(x)-\beta^{k+1}(x)\Big)f\big(\beta^{k}(x)\big)\,.
\end{equation}
Thus,
\begin{equation}\label{beta-integral-a-b-explicitly}
\int_{a}^{b}f\,{\rm d}_{\beta}=\sum_{k=0}^{+\infty}\Big(\beta^k(b)-\beta^{k+1}(b)\Big)f\big(\beta^{k}(b)\big)-
\sum_{k=0}^{+\infty}\Big(\beta^k(a)-\beta^{k+1}(a)\Big)f\big(\beta^{k}(a)\big)\,.
\end{equation}
If the infinite sum in the right side of (\ref{beta-integral-s0-x}) is convergent then
we say that the function $\,f\,$ is $\,\beta-$integrable in $\,[s_0,x].$
The $\,\beta-$integral in the left side of (\ref{beta-integral-a-b}) is well defined
provided that at least one of the $\,\beta-$integrals in the right side is finite and
we say that $\,f\,$ is $\,\beta-$integrable in $\,[a,b]\,$ if it is both
$\,\beta-$integrable in $\,[s_0,a]\,$ and in $\,[s_0,b].$

\noindent As important particular cases, one obtains the Jackson-Thomae-N\"orlund integral
when $\,\beta(t)=qt+\omega\,$ with $\,0<q<1\,$ and $\,\omega\geq 0$.
The Jackson $\,q-$integral corresponds to $\,\omega=0\,$ in the previous case.
Its fixed points are given by $\,s_0=\frac{\omega}{1-q}$ and $\,s_0=0\,$, respectively.

\subsection{Properties of the $\beta-$derivative and of the $\beta-$integral}

We go back to the the definition of the $\,\beta-$derivative operator
(\ref{beta-derivative}).

\noindent Notice that if $\,f\,$ is differentiable at a point $\,t\in I$, then
$$
\lim_{\beta(t)\rightarrow t}D_{\beta}[f](t)=f^\prime(t)\,,
$$
hence $\,D_{\beta}\,$ is a $\,beta-$analogue of the standard derivative operator.

\noindent The $\,\beta-$derivative satisfies properties which may be regarded as
$\,\beta-$analogues of the corresponding properties for the usual derivative.
For instance, the quantum operator (\ref{beta-derivative}) is linear, i.e.,
\begin{equation*}
D_{\beta}[\alpha f+\beta g](t)=\alpha D_{\beta}[f](t)+\beta D_{\beta}[g](t)\,,
\end{equation*}
where $\,\alpha\,$ and $\,\beta\,$ are any real or complex numbers, and satisfies the
following $\,\beta-$product rule: for $\,t\in I,$
\begin{equation}\label{beta-product-rule}
\begin{array}{lll}
D_{\beta}[f\cdot g](t) & = & D_{\beta}[f](t)\cdot g(t)+f\big(\beta(t)\big)\cdot D_{\beta}[g](t) \\ [0.8em]
 & = & D_{\beta}[g](t)\cdot f(t)+g\big(\beta(t)\big)\cdot D_{\beta}[f](t)
 \end{array}
\end{equation}
if $\,f\,$ and $\,g\,$ are $\,\beta-$differentiable in $\,I$.
Also, $\,f\,$ will be the constant function such that $\,f(t)=f(s_0)\,$ for all $\,t\in I\,$
whenever $\,D_{\beta}[f](t)=0\,$ for all $\,t\in I\,$. For these and other properties of the
general quantum difference operator $\,D_{\beta}\,$ see \cite{HSSA:2015}.
These equalities hold for all $\,t\neq s_0$,
and also for $\,t=s_0\,$ whenever $\,f^\prime(s_0)\,$ and $\,g^\prime(s_0)$ exist.

\subsubsection{The fundamental theorem of $\beta-$calculus}

The following statement of the $\,\beta-$analogue of the fundamental theorem of
calculus can be found in \cite{C:2020}.

\vspace{0.5em}
\textbf{Theorem B}.
Let $\,\beta:I\to I\,$ be a function satisfying the conditions described in subsection
\ref{Subsect-2.1}.
Fix $\,a,b\in I$ and let $\,f:I\rightarrow\mathbb{C}\,$ be a function such that
$\,D_{\beta}[f]\in\mathscr{L}_{\beta}^1[a,b]$. Then:
\begin{enumerate}
\item[{\rm (i)}] The equality
\begin{equation*}
\int_{a}^bD_{\beta}[f]\,{\rm d}_{\beta}=
\left[f(s)-\lim_{k\rightarrow+\infty}f\big(\beta^{k}(s)\big)\right]_{s=a}^b
\end{equation*}
holds, provided the involved limits exist.
\item[{\rm (ii)}] In addition, assuming that $\,a<s_0<b,$ if $\,f\,$ has a discontinuity
of first kind at $\,s_0\,$ then
$$
\int_{a}^bD_{\beta}[f]\,{\rm d}_{\beta}=f(b)-f(a)-\Big(f(s_0^+)-f(s_0^-)\Big).
$$
Of course, if $\,f\,$ is continuous at $\,s_0\,$ then
$$
\int_{a}^bD_{\beta}[f]\,{\rm d}_{\beta}=f(b)-f(a)\,.
$$
\end{enumerate}

\subsubsection{The $\beta-$integration by parts formula}

Now we state the $\,\beta-$analogue of the integration by parts formula \cite{C:2020}.

\vspace{0.5em}
\textbf{Theorem C}.
Let $\,\beta:I\to I\,$ be a function satisfying the conditions described in subsection
\ref{Subsect-2.1}.
Fix $\,a,b\in I$ and two functions $\,f:I\rightarrow\mathbb{C}\,$ and
$\,g:I\rightarrow\mathbb{C}.$ Then:
$$
\int_{a}^b f\cdot D_{\beta}[g]\,{\rm d}_{\beta}=
\left[(f\cdot g)(s)-\lim_{k\rightarrow+\infty}(f\cdot g)\big(\beta^k(s)\big)\right]_{s=a}^b
-\int_{a}^b\big(g\circ\beta\big)\cdot D_{\beta}[f]\,{\rm d}_{\beta}
$$
holds, provided $\,f,g\in\mathscr{L}_{\beta}^1[a,b]$, $\,D_{\beta}[f]\,$ and $\,D_{\beta}[g]\,$
are bounded in $\,[a,b]_{\beta}$, and the limits exist.

If, in addition, $\,f\,$ and $\,g\,$ are continuous at $\,s_0\,$ and $\,a<s_0<b,$ then
$$
\int_{a}^b f\cdot D_{\beta}[g]\,{\rm d}_{\beta}=
\Big[ f\cdot g\Big]_{a}^b-\int_{a}^b \big(g\circ\beta\big)\cdot D_{\beta}[f]\,{\rm d}_{\beta}\,.
$$

\subsection{The spaces $\,\mathscr{L}_{\beta}^p[a,b]\,$ and $\,L_{\beta}^p[a,b]$}

\subsubsection{The space $\mathscr{L}_{\beta}^p[a,b]$}

For $a,b\in I$, we will denote by $\mathscr{L}_{\beta}^p[a,b]$ the set of functions
$f:I\rightarrow\mathbb{\mathbb{C}}$ such that $|f|^p$ is $\beta-$integrable in $[a,b]$, i.e.,
$$\mathscr{L}_{\beta}^p[a,b]=\left\{f:I\rightarrow\mathbb{\mathbb{C}}\;\Big|\;
\int_{a}^b|f|^p{\rm d}_{\beta}<\infty\right\}\;.$$
We also set
$$\mathscr{L}_{\beta}^\infty[a,b]=\left\{f:I\rightarrow\mathbb{\mathbb{C}}\;\Big|\;
\sup_{k\in\mathbb{N}_0}\Big\{\big|f\big(\beta^k(a)\big)\big|,\big|f\big(\beta^k(b)\big)\big|\Big\}<\infty\right\}\;.$$
It was proved in \cite[Corollary 3.4]{C:2020} that if $\,a\leq
s_0\leq b\,$ and $\,1\leq p\leq\infty$, then the set
$\,\mathscr{L}_{\beta}^p[a,b]$, with the usual operations of
addition of functions and multiplication of a function by a number
(real or complex), becomes a linear space over $\,\mathbb{K}$.

\subsubsection{The space $L_{\beta}^p[a,b]$}

For $\,f,g\in\mathscr{L}_{\beta}^p[a,b]$, writing $\,f\sim g\,$ when
\begin{equation}\label{EquivRel}
f\big(\beta^k(a)\big)=g\big(\beta^k(a)\big)\quad \mbox{and}\quad f\big(\beta^k(b)\big)=g\big(\beta^k(b)\big)
\end{equation}
holds for all $\,k=0,1,2,\cdots$, i.e., we say that $\,f\sim g\,$ if $\,f=g\,$ in
$\,[a,b]_{\beta}$.
Clearly, $\,\sim\,$ defines an equivalence relation in
$\,\mathscr{L}_{\beta}^p[a,b]$.
We will represent by $\,L_{\beta}^p[a,b]\,$ the corresponding quotient set:
$$
L_{\beta}^p[a,b]:=\mathscr{L}_{\beta}^p[a,b]\big/\sim\;.
$$
Also in \cite{C:2020} it was proved the following theorems.

\textbf{Theorem D}.
If $\,a\leq s_0\leq b\,$ and $\,1\leq p\leq \infty\,$ then
$\,L_{\beta}^{p}[a,b]\,$ is a normed linear space over $\,\mathbb{R}\,$ or $\,\mathbb{C}\,$ with norm
\begin{equation}\label{beta-p-norm}
\|f\|_{L_{\beta}^p[a,b]}:=\left\{
\begin{array}{ccl}
\displaystyle\left(\int_a^b|f|^p\,{\rm d}_{\beta}\right)^{\frac 1 p}&{\rm if}& 1\leq p<\infty\; ; \\
\rule{0pt}{1.5em}
\displaystyle\sup_{k\in\mathbb{N}_0}\,\Big\{\left|f\big(\beta^k(a)\big)\right|,\left|f\big(\beta^k(b)\big)\right|\Big\} &{\rm if}& p=\infty\;.
\end{array}\right.
\end{equation}

\vspace{0.5em}
In the usual way, here $\,f\in\mathscr{L}_{\beta}^p[a,b]\,$ denotes any representative
of the correspondent class in $\,L_{\beta}^p[a,b]\,$ where it belongs. Notice that,
in view of (\ref{beta-integral-a-b-explicitly}) and (\ref{EquivRel}), the definition of
the norm $\,\|f\|_{L_{\beta}^p[a,b]}\,$ is independent of the chosen representative.

\vspace{0.5em}
\textbf{Theorem E}.
If $\,a\leq s_0\leq b\,$ and $\,1\leq p\leq \infty$, then the following holds:

{\rm (i)} $L_{\beta}^{p}[a,b],$ endowed with the norm (\ref{beta-p-norm}), is a Banach
space for $\,1\leq p\leq\infty$, which is separable if $\,1\leq p<\infty\,$ and reflexive if
$\,1<p<\infty$.

{\rm (ii)} $L_{\beta}^2[a,b]$ is a Hilbert space with inner product
\begin{equation}\label{beta-inner-product}
\langle f,g\rangle_{\beta}:=\int_a^bf\,\overline{g}\,{\rm d}_{\beta}\; ,\quad f,g\in L_{\beta}^2[a,b]\; .
\end{equation}

\subsection{The $\,\beta$-exponential and $\,\beta$-trigonometric functions}\label{beta-exponential-trigonometric}
The following $\,\beta$-analogues of the exponential and trigonometric functions,
as well as some of its properties, were introduced in \cite{HSS:2017}.

It is assumed that $\,p:I\to\mathbb{C}\,$ is a continuous function
at $\,s_0\,$. The $\,\beta$-exponential functions are defined by
\begin{equation}\label{beta-exp}
\mbox{e}_{p,\beta}(t)=\frac{1}{\displaystyle\prod_{k=0}^{\infty}
\left[1-p\left(\beta^k(t)\right)\Big(\beta^k(t)-\beta^{k+1}(t)\Big)\right]}
\end{equation}
and
\begin{equation}\label{beta-Exp}
\mbox{E}_{p,\beta}(t)=\prod_{k=0}^{\infty}
\left[1+p\left(\beta^k(t)\right)\Big(\beta^k(t)-\beta^{k+1}(t)\Big)\right]\,.
\end{equation}
We notice that, under the assumptions on the functions $\,p\,$ and $\,\beta\,$, the
infinite products of (\ref{beta-exp}) and of (\ref{beta-Exp}) are both convergent
\cite[page 28]{HSS:2017}.

On the other hand, the trigonometric functions are defined by
\begin{equation}\label{beta-trigonometric}
\mbox{cos}_{p,\beta}(t)=\frac{\mbox{e}_{ip,\beta}(t)+\mbox{e}_{-ip,\beta}(t)}{2}\;,\quad
\mbox{sin}_{p,\beta}(t)=\frac{\mbox{e}_{ip,\beta}(t)-\mbox{e}_{-ip,\beta}(t)}{2}
\end{equation}
and
\begin{equation}\label{beta-Trigonometric}
\mbox{Cos}_{p,\beta}(t)=\frac{\mbox{E}_{ip,\beta}(t)+\mbox{E}_{-ip,\beta}(t)}{2}\;,\quad
\mbox{Sin}_{p,\beta}(t)=\frac{\mbox{E}_{ip,\beta}(t)-\mbox{E}_{-ip,\beta}(t)}{2}\,.
\end{equation}
It is known that
\begin{equation*}
\mbox{e}_{p,\beta}(t)=\frac{1}{\mbox{E}_{-p,\beta}(t)}\;,\quad \mbox{e}_{p,\beta}(s_0)=1\;,
\quad \mbox{E}_{p,\beta}(s_0)=1
\end{equation*}
and
\begin{equation}\label{trigonometric-beta-derivatives}
\left\{\begin{array}{l}
D_{\beta}\,\mbox{e}_{p,\beta}(t)=p(t)\,\mbox{e}_{p,\beta}(t)\,,\quad
D_{\beta}\,\mbox{E}_{p,\beta}(t)=p(t)\,\mbox{E}_{p,\beta}\big(\beta(t)\big)\,, \\ [1em]
D_{\beta}\,\mbox{sin}_{p,\beta}(t)=p(t)\,\mbox{cos}_{p,\beta}(t)\,,\quad
D_{\beta}\,\mbox{cos}_{p,\beta}(t)=-p(t)\,\mbox{sin}_{p,\beta}(t)\,, \\ [1em]
D_{\beta}\,\mbox{Sin}_{p,\beta}(t)=p(t)\,\mbox{Cos}_{p,\beta}(t)\,,\quad
D_{\beta}\,\mbox{Cos}_{p,\beta}(t)=-p(t)\,\mbox{Sin}_{p,\beta}(t)\,.
\end{array}\right.
\end{equation}

\section{The $\beta-$Sturm-Liouville problem}\label{Sect-3}

\subsection{Simple formulation of the classical Sturm-Liouville problem}

The simplest formulation of the classical Sturm-Liouville problem with separate
type conditions, is the following:
\begin{equation}\label{SLP-1}
\left\{\begin{array}{cc}
-y^{\prime\prime}+\nu(t)y=\lambda y\,,& -\infty<a\leq t\leq b<\infty \\ [0.5em]
a_{11}y(a)+a_{12}y^{\prime}(a)=0\,,& \\ [0.5em]
a_{21}y(b)+a_{12}y^{\prime}(b)=0\,.&
\end{array}\right.,
\end{equation}
where $\,\nu(\cdot)\,$ is a real valued continuous function on $\,[a,b]\,$,
$\,\lambda\in\mathbb{C}\,$ and $\vert a_{i1}\vert+\vert a_{i2}\vert\neq 0\,$,
$\,i=1,2\,$. The major properties of the Sturm-Liouville problem  (\ref{SLP-1})
are that there are only a countable number of real numbers (eigenvalues)
$\,\lambda_1<\lambda_2<\lambda_3<\cdots\,$, with $\lambda_n\to\infty\,$, such
that (\ref{SLP-1}) has nontrivial solutions $\varphi_1,\varphi_2,\varphi_3\,\cdots\,$
(eigenfunctions). Moreover, to each eigenvalue $\,\lambda_n\,$, up to a multiplicative
constant, there corresponds only one solution $\,\varphi_n\,$. Thus, the eigenvalues
are geometrically simple and they are also algebraically simple since they are
simple zeros of the characteristic function associated with (\ref{SLP-1}). Also,
the set $\,\{\varphi_n\}_{n=0}^{\infty}\,$ is an orthonormal basis of $\,L^2(a,b)\,$.
Regarding all of these aspects see \cite{T:1962}.

Annaby and Mansour \cite{AM:2005, AM:2012}, after Exton \cite{E:1983, E:1992},
developed a $q-$Sturm-Liouville theory in the regular setting.
Later, together with Makharesh \cite{AHM:2018}, they make progress in the direction of
a Sturm-Liouville theory in the regular setting, with separate-type boundary conditions,
associated with the $\,D_{q,\omega}\,$ operator (\ref{Hahn-operator}).

In the following, we will establish conditions to develop a similar Sturm-Liouville
theory associated with the general quantum operator $\,D_{\beta}\,$ defined by
(\ref{beta-derivative}).

\subsection{Further properties for the $\beta$-difference operator and
the $\beta$-integral}

In the next lemmas, $\,b\,$ is a fixed real parameter.
\begin{lem}\label{M-V}
Let $\,\beta:I\to I\,$ be a function satisfying the conditions described in subsection
\ref{Subsect-2.1}, being $\,s_0\,$ its unique fixed point, and $\,f:I\rightarrow\mathbb{C}\,$
a real or complex function. Then

$\displaystyle\int_{s_0}^{b}f\big(\beta(t)\big)d_{\beta}t=
\int_{s_0}^{\beta{(b)}}f(u)\left(D_{\beta}\beta^{-1}\right)(u)\,d_{\beta}u\,$.
\end{lem}
\begin{proof}
The proof is straightforward since, by definition, the right side \\
$\;\displaystyle\int_{s_0}^{\beta(b)}f(u)\left(D_{\beta}\beta^{-1}\right)(u)\,d_{\beta}u\,$
equals

$\displaystyle\sum_{k=0}^{\infty}
f\Big(\beta^{k}\left(\beta(b)\right)\Big)
\frac{\beta^{-1}\Big(\beta^{k}\big(\beta(b)\big)\Big)-\beta^{k}\big(\beta(b)\big)\Big)}{
\beta^{k}\big(\beta(b)\big)\Big)-\beta\Big(\beta^{k}\big(\beta(b)\big)\Big)\Big)}
\Big(\beta^{k}\big(\beta(b)\big)-\beta^{k+1}\big(\beta(b)\big)\Big)\,$,

\noindent
which becomes, after simplification,

$\displaystyle\sum_{k=0}^{\infty}
f\Big(\beta\left(\beta^{k}(b)\right)\Big)\Big(\beta^{k}(b)-\beta^{k+1}(b)\Big)=
\displaystyle\int_{s_0}^{b}f\big(\beta(t)\big)d_{\beta}t\,.$
\end{proof}

In the following, in order to simplify the notation of the inner-product
(\ref{beta-inner-product}) we will consider $\,\langle\,\cdot\,,\,\cdot\,\rangle\,$
rather than $\,\langle\,\cdot\,,\,\cdot\,\rangle_{\beta}\,.$ Of course that when we
write the interval $\,(s_0,b)\,$ we admit that $\,s_0\leq b\,$.
If $\,b<s_0\,$ then one must replace that interval by $\,(b,s_0)\,$.
The corresponding proofs follow exactly the same steps.
\begin{lem}\label{beta-Lemma3.2}
Let $\,\beta:I\to I\,$ be a function satisfying the conditions described in subsection
\ref{Subsect-2.1}, being $\,s_0\,$ its unique fixed point, and
$\,f,g\in L_{\beta}^{2}\big(s_0,b\big)\,$ be both continuous
functions at $\,s_0\,$. Then, for $\,t\in\big(s_0,b\big]\,$ we have
\begin{enumerate}
\item[(i)]
$\,\displaystyle\left(D_{\beta}f\right)\left(\beta^{-1}(t)\right)=
\left(D_{\beta^{-1}}f\right)(t)\,;$

\vspace{0.5em}
\item[(ii)] $\,\langle D_{\beta}f\,,\,g\rangle=
f(b)g\overline{\big(\beta^{-1}(b)\big)}-f(s_0)g\overline{\big(s_0\big)}+
\langle f\,,\,-D_{\beta}\beta^{-1}D_{\beta^{-1}}g\rangle\,;$

\vspace{0.5em}
\item[(iii)] $\,\langle -D_{\beta}\beta^{-1}D_{\beta^{-1}}f\,,\,g\rangle=
f\big(s_0\big)\overline{g\big(s_0\big)}-f\Big(\beta^{-1}(b)\Big)\overline{g(b)}
+\langle f\,,\,D_{\beta}g\rangle\,.$
\end{enumerate}
\end{lem}
\begin{proof}
The proof of (i) is trivial.
Let's prove (ii):

\noindent By the $\beta$-integration by parts theorem (Theorem C) we have
\[
\begin{array}{l}
\langle D_{\beta}f\,,\,g\rangle\:=\:\displaystyle
\int_{s_0}^{b}\Big(D_{\beta}f\Big)(t)\overline{g(t)}d_{\beta}t \\ [0.5em]
\hspace{4.6em}=\displaystyle\:f(b)\overline{g(b)}-f(s_0)\overline{g(s_0)}-
\int_{s_0}^{b}f\big(\beta(t)\big)\overline{\Big(D_{\beta}g\Big)(t)}d_{\beta}t\,.
\end{array}
\]
Considering $\,u=\beta(t)\,$ then, by Lemma \ref{M-V}, one obtains
\[
\begin{array}{l}
\langle D_{\beta}f\,,\,g\rangle =
\displaystyle\:f(b)\overline{g(b)}-f(s_0)\overline{g(s_0)} \hspace{5em} \\ [0.3em]
\hspace{5em}\displaystyle-\int_{s_0}^{\beta(b)}\!f(u)\Big(D_{\beta}\beta^{-1}\Big)(u)
\overline{\Big(D_{\beta}g\Big)\left(\beta^{-1}(u)\right)}d_{\beta}u\,,
\end{array}
\]
which, by (i), becomes
\begin{equation}\label{break}
\langle D_{\beta}f\,,\,g\rangle=
\displaystyle f(b)\overline{g(b)}-f(s_0)\overline{g(s_0)}-
\int_{s_0}^{\beta(b)}f(u)D_{\beta}\beta^{-1}(u)
\overline{D_{\beta^{-1}}g(u)}d_{\beta}u.
\end{equation}
Writing (see (iv) of Lemma 3.5 in \cite{HSSA:2015})
\[
\begin{array}{l}
\displaystyle\int_{s_0}^{\beta(b)}f(u)D_{\beta}\beta^{-1}(u)
\overline{D_{\beta^{-1}}g(u)}d_{\beta}u=
\int_{s_0}^{b}f(u)D_{\beta}\beta^{-1}(u)
\overline{D_{\beta^{-1}}g(u)}d_{\beta}u \\ [0.8em]
\hspace{10em}\displaystyle+\int_{b}^{\beta(b)}f(u)D_{\beta}\beta^{-1}(u)
\overline{D_{\beta^{-1}}g(u)}d_{\beta}u\,.
\end{array}
\]
then, by Corollary 3.7 of \cite{HSSA:2015},
\[
\begin{array}{l}
\displaystyle\int_{s_0}^{\beta(b)}f(u)D_{\beta}\beta^{-1}(u)
\overline{D_{\beta^{-1}}g(u)}d_{\beta}u=\int_{s_0}^{b}f(u)D_{\beta}\beta^{-1}(u)
\overline{D_{\beta^{-1}}g(u)}d_{\beta}u \\ [0.8em]
\hspace{10em}+\big(\beta(b)-b\big)f(b)D_{\beta}\beta^{-1}(b)
\overline{D_{\beta^{-1}}g(b)}\,.
\end{array}
\]
Introducing this last identity in (\ref{break}) one gets
\[
\begin{array}{l}
\displaystyle\langle D_{\beta}f\,,\,g\rangle =
\displaystyle\:f(b)\overline{g(b)}-f(s_0)\overline{g(s_0)}
-\int_{s_0}^{b}f(u)D_{\beta}\beta^{-1}(u)
\overline{D_{\beta^{-1}}g(u)}d_{\beta}u \\ [0.8em]
\hspace{10em}-\big(\beta(b)-b\big)f(b)D_{\beta}\beta^{-1}(b)
\overline{D_{\beta^{-1}}g(b)}\,.
\end{array}
\]
Since
\[
\begin{array}{l}
\displaystyle\:f(b)\overline{g(b)}-f(s_0)\overline{g(s_0)}
-\big(\beta(b)-b\big)f(b)\Big(D_{\beta}\beta^{-1}\Big)(b)
\overline{\Big(D_{\beta^{-1}}g\Big)(b)}= \\ [0.8em]
\hspace{10em}\displaystyle\:
f(b)g\overline{\big(\beta^{-1}(b)\big)}-f(s_0)g\overline{\big(s_0\big)}\,.
\end{array}
\]
we obtain
\[
\begin{array}{l}
\displaystyle\langle D_{\beta}f\,,\,g\rangle =
\displaystyle\:f(b)g\overline{\big(\beta^{-1}(b)\big)}-f(s_0)g\overline{\big(s_0\big)}
-\langle f\,,\,D_{\beta}\beta^{-1}D_{\beta^{-1}}g\rangle\,,
\end{array}
\]
which proves (ii).

\vspace{0.5em}
Finally, (iii) is a consequence of (ii).
\end{proof}

\noindent Taking into account the $\beta$-inner-product (\ref{beta-inner-product}),
we thus have the following corollary. It is a direct consequence of (ii), Lemma
\ref{beta-Lemma3.2}, therefore its proof will be omitted. Here $\,\beta\,$ satisfies
the same conditions admitted in Lemmas \ref{M-V} and \ref{beta-Lemma3.2}.
\begin{cor}\label{beta-Lemma3.2-a-b}
Let $\,a\leq s_0\leq b\,$ and $\,f,g\in L_{\beta}^{2}(a,b)\,$ be both
continuous functions at $\,s_0\,$. Then,
\begin{enumerate}
\item[(i)]
$\,\displaystyle\left(D_{\beta}f\right)\left(\beta^{-1}(t)\right)=
\left(D_{\beta^{-1}}f\right)(t)\,;$

\vspace{0.5em}
\item[(ii)] $\,\langle D_{\beta}f\,,\,g\rangle=
f(b)g\overline{\big(\beta^{-1}(b)\big)}-f(a)\overline{g\big(\beta^{-1}(a)\big)}+
\langle f\,,\,-D_{\beta}\beta^{-1}D_{\beta^{-1}}g\rangle\,.$
\end{enumerate}
\end{cor}

\subsection{The $\beta$-eigenvalue problem}

Consider the following $\beta$-Sturm-Liouville problem ($\,\beta$-SLP) in
$\,L_{\beta}^2\left(s_0,b\right)\,$ defined by the $\,\beta$-difference equation
\begin{equation}\label{ell-beta-Operator}
\ell_{\beta}y(t):=-D_{\beta}\beta^{-1}D_{\beta^{-1}}D_{\beta}y(t)+r(t)y(t)=
\lambda\,y(t)\,,
\end{equation}
with $\,s_0\leq t\leq b<\infty\,$, $\,\lambda\in\mathbb{C}\,$, and the boundary
conditions
\begin{equation}\label{Boundary-Conditions}
\left\{\begin{array}{l}
\displaystyle a_1y(s_0)+a_2D_{\beta^{-1}}y\left(s_0\right)=0 \\ [0.5em]
\displaystyle b_1y(b)+b_2D_{\beta^{-1}}y(b)=0\,.
\end{array}\right.
\end{equation}
We assume that $\,\beta:I\to I\,$ is a function satisfying the conditions described in
subsection \ref{Subsect-2.1}, being $\,s_0\,$ its unique fixed point, and also
assume that $\,r(t)\,$ is a real valued continuous function on $\,[s_0,b]\,$
and $\,\vert a_1\vert+\vert a_2\vert\neq 0\neq \vert b_1\vert+\vert b_2\vert$.

\vspace{0.5em}
The operator $\,\ell_{\beta}\,$ (\ref{ell-beta-Operator}) satisfies the following
$\,\beta$-Lagrange's identity.
\begin{thm}\label{beta-Lagrange-Identity-Theorem}
If $\,y\,,\,z\,\in L_{\beta}^2\left(s_0,b\right)\,$ then
\begin{equation}\label{beta-Lagrange-Identity}
\int_{s_0}^{b}\left[\Big(\ell_{\beta}y(t)\Big)\overline{z(t)}-
y(t)\overline{\Big(\ell_{\beta}z(t)\Big)}\right]d_{\beta}u=
[y,\overline{z}](b)-[y,\overline{z}]\left(s_0\right)\,,
\end{equation}
where
\begin{equation}\label{beta-1-Wronskian}
\,[y,z](t)=y(t)D_{\beta^{-1}}z(t)-z(t)D_{\beta^{-1}}y(t)
\end{equation}
\end{thm}
\begin{proof}
Consider $\,y\,,\,z\,\in L_{\beta}^2\left(s_0,b\right)\,$.

\noindent On one hand, using (iii), Lemma \ref{beta-Lemma3.2}, with
$\,f(t)=D_{\beta}y(t)\,$ and $\,g(t)=z(t)\,$, we have
$$\langle -D_{\beta}\beta^{-1}D_{\beta^{-1}}D_{\beta}y,z\rangle=
-D_{\beta}y\left(\beta^{-1}(b)\right)\overline{z(b)}+
D_{\beta}y\left(s_0\right)\overline{z\left(s_0\right)}
+\langle D_{\beta}y\,,D_{\beta}z\rangle.$$
By (i) of Lemma \ref{beta-Lemma3.2} we then obtain
\begin{equation}\label{Eq3.16}
\,\langle -D_{\beta}\beta^{-1}D_{\beta^{-1}}D_{\beta}y,z\rangle=
-D_{\beta^{-1}}y(b)\,\overline{z(b)}+
D_{\beta^{-1}}y\left(s_0\right)\overline{z\left(s_0\right)}
+\langle D_{\beta}y,D_{\beta}z\rangle.
\end{equation}
On the other hand, using (ii), Lemma \ref{beta-Lemma3.2}, with $\,f(t)=y(t)\,$ and
$\,g(t)=D_{\beta}z(t)\,$, we get
$$\langle D_{\beta}y\,,D_{\beta}z\rangle =
y(b)\overline{D_{\beta}z\left(\beta^{-1}(b)\right)}-
y\left(s_0\right)\overline{D_{\beta}z\left(s_0\right)}+
\langle y,-D_{\beta}\beta^{-1}D_{\beta^{-1}}D_{\beta}z\rangle,$$
which becomes, by (i) of Lemma \ref{beta-Lemma3.2},
\begin{equation}\label{Eq3.17}
\langle D_{\beta}y\,,D_{\beta}z\rangle =
y(b)\overline{D_{\beta^{-1}}z(b)}-
y\left(s_0\right)\overline{D_{\beta^{-1}}z\left(s_0\right)}+
\langle y,-D_{\beta}\beta^{-1}D_{\beta^{-1}}D_{\beta}z\rangle.
\end{equation}
Combining (\ref{Eq3.16}) with (\ref{Eq3.17}) it results
\[
\begin{array}{l}
\displaystyle\,\langle -D_{\beta}\beta^{-1}D_{\beta^{-1}}D_{\beta}y\,,\,z\rangle=
y(b)\overline{\left(D_{\beta^{-1}}z\right)(b)}-
\left(D_{\beta^{-1}}y\right)(b)\,\overline{z(b)} \\ [0.8em]
\hspace{4em} \displaystyle
-y\left(s_0\right)\overline{\left(D_{\beta^{-1}}z\right)\left(s_0\right)}
+D_{\beta^{-1}}y\left(s_0\right)\overline{z\left(s_0\right)}
+\langle y,-D_{\beta}\beta^{-1}D_{\beta^{-1}}D_{\beta}z\rangle
\end{array}
\]
which is equivalent to
\begin{equation}\label{beta-Lagrange-Identity-2}
\displaystyle\,\langle -D_{\beta}\beta^{-1}D_{\beta^{-1}}D_{\beta}y\,,\,z\rangle=
[y,\overline{z}](b)-[y,\overline{z}]\left(s_0\right)
+\langle y,-D_{\beta}\beta^{-1}D_{\beta^{-1}}D_{\beta}z\rangle\,.
\end{equation}
Now we easily derive the $\,\beta$-Lagrange's identity
(\ref{beta-Lagrange-Identity}): using (\ref{beta-Lagrange-Identity-2}) and
the fact that $\,r(t)\,$ is real we have
\[
\begin{array}{lll}
\displaystyle\,\langle\,\ell_{\beta}y\,,\,z\,\rangle & = &
\langle\,D_{\beta}\beta^{-1}D_{\beta^{-1}}D_{\beta}y(t)+r(t)y(t)\,,\,z(t)\,\rangle \\ [0.7em]
& = & \langle\, D_{\beta}\beta^{-1}D_{\beta^{-1}}D_{\beta}y(t)\,,\,z(t)\,\rangle +
\langle\, r(t)y(t)\,,\,z(t)\,\rangle \\ [0.7em]
& = & [y,\overline{z}](b)-[y,\overline{z}]\left(s_0\right)
+\langle\, y(t),-D_{\beta}\beta^{-1}D_{\beta^{-1}}D_{\beta}z(t)\,\rangle \\ [0.7em]
& & +\langle \,y(t)\,,\,r(t)z(t)\,\rangle \\ [0.7em]
& = & [y,\overline{z}](b)-[y,\overline{z}]\left(s_0\right)
+\langle\, y(t),-D_{\beta}\beta^{-1}D_{\beta^{-1}}D_{\beta}z(t)+r(t)z(t)\,\rangle \\ [0.7em]
& = & [y,\overline{z}](b)-[y,\overline{z}]\left(s_0\right)+\langle\, y\,,\,\ell_{\beta}z\,\rangle
\end{array}
\]
\end{proof}
The following corollary follows now easily.
\begin{cor}\label{Self-Adjoint-Operator}
The $\,\beta$-Sturm Liouville eigenvalue problem (\ref{ell-beta-Operator}-
\ref{Boundary-Conditions}) is self-adjoint in $L_{\beta}^2\left(s_0,b\right)$
\big(i.e., $\ell_{\beta}$ is self-adj. in
$\big\{y\in L_{\beta}^2\left(s_0,b\right):\:y
\mbox{\,satisfies\,(\ref{Boundary-Conditions})}\big\}$\big).
\end{cor}
\begin{proof}
Let $\,y\,$ and $\,z\,$ satisfy the boundary conditions (\ref{Boundary-Conditions}).
\begin{enumerate}
\item[(i)] If $\,a_2\neq 0\,$ then, from (\ref{beta-1-Wronskian}),
$$[y,\overline{z}]\left(s_0\right)=y\left(s_0\right)\overline{
\left(-\frac{a_1}{a_2}\,z\left(s_0\right)\right)}+\frac{a_1}{a_2}\,y\left(s_0\right)
\overline{z\left(s_0\right)}=0\,.$$
\item[(ii)] If $\,a_2=0\,$ then, since $\,|a_1|+|a_2|\neq 0\,$, one must have $\,a_1\neq 0\,$
which, by (\ref{Boundary-Conditions}), implies that $\,y\left(s_0\right)=0\,$.

\noindent By similar arguments it follows that $\,z\left(s_0\right)=0\,$.
Thus, $\,[y,\overline{z}]\left(s_0\right)=0\,.$
\end{enumerate}
Arguing as above one proves also that $\,[y,\overline{z}](b)=0\,$ thus,
$$\langle\,\ell_{\beta}y\,,\,z\,\rangle=\langle\, y\,,\,\ell_{\beta}z\,\rangle\,.$$
\end{proof}
\begin{rem}
As a consequence of (\ref{beta-Lagrange-Identity-2}), under the boundary conditions
(\ref{Boundary-Conditions}), of course we also have
\begin{equation*}\label{beta-Lagrange-Identity-3}
\displaystyle\,\langle -D_{\beta}\beta^{-1}D_{\beta^{-1}}D_{\beta}y\,,\,z\rangle=
\langle y,-D_{\beta}\beta^{-1}D_{\beta^{-1}}D_{\beta}z\rangle\,.
\end{equation*}
\end{rem}

\begin{thm}\label{beta-Lemma3.5}
All eigenvalues of the problem (\ref{ell-beta-Operator})-(\ref{Boundary-Conditions})
are real. Eigenfunctions corresponding to different eigenvalues are orthogonal.
\end{thm}
\begin{proof}
We separate the proof in two parts: (i) and (ii).

\noindent
(i) First we show that the eigenfunctions are real.

\noindent Let $\,\lambda_0\,$ be an eigenvalue and $\,y_0(t)\,$ be a corresponding
eigenfunction. Since
$$\,\ell_{\beta}\,y_0(t)=\lambda_0\,y_0(t)\,,\quad
\overline{\ell_{\beta}\,y_0(t)}=\overline{\lambda_0\,y_0(t)}$$
and, by Corollary \ref{Self-Adjoint-Operator},
$$\,\int_{s_0}^{b}\ell_{\beta}\,y_0(t)\,\overline{y_0(t)}d_{\beta}t=
\int_{s_0}^{b}y_0(t)\,\overline{\ell_{\beta}\,y_0(t)}d_{\beta}t\,,$$
then,
$$\,\lambda_0\int_{s_0}^{b}\left|y_0(t)\right|^2\,d_{\beta}t=
\overline{\lambda_0}\int_{s_0}^{b}\left|y_0(t)\right|^2\,d_{\beta}t\,,$$
thus
$$
\,\left(\lambda_0-\overline{\lambda_0}\right)\int_{s_0}^{b}\left|y_0(t)\right|^2\,d_{\beta}t=0\,.
$$
Since $\,y_0(t)\,$ is an eigenfunction then $\,\lambda_0=\overline{\lambda_0}\,$.

\vspace{0.5em}
(ii) Finally, we show that eigenfunctions corresponding to
different eigenvalues are orthogonal.

\noindent Let $\,\lambda_1\neq\lambda_2\,$ be distinct eigenvalues with eigenfunctions
$\,\phi_1\,$ and $\,\phi_2\,$, respectively.
From Corollary \ref{Self-Adjoint-Operator} 
and because the eigenvalues are real one have
$$\lambda_1\int_{s_0}^b\phi_1(t)\overline{\phi_2(t)}d_{\beta}t=
\lambda_2\int_{s_0}^b\phi_1(t)\overline{\phi_2(t)}d_{\beta}t\,.$$
Since $\,\lambda_1\neq\lambda_2\,$ the orthogonality follows.
\end{proof}

Now we generalize these results in the following way, where $\,\beta\,$ is a function
satisfying the conditions introduced in the very beginning of subsection \ref{Subsect-2.1}.

If $\,a\leq s_0\leq b\,$ then we may consider the $\,\beta$-SLP in
$\,L_{\beta}^2\left(a,b\right)\,$ defined by the $\,\beta$-difference equation
\begin{equation}\label{ell-beta-Operator-a-b}
\ell_{\beta}^{a,b}y(t):=-D_{\beta}\beta^{-1}D_{\beta^{-1}}D_{\beta}y(t)+r(t)y(t)=
\lambda\,y(t)\,,
\end{equation}
with $\,-\infty<a\leq s_0\leq b<\infty\,$, $\,\lambda\in\mathbb{C}\,$, and the boundary
conditions
\begin{equation}\label{Boundary-Conditions-a-b}
\left\{\begin{array}{l}
\displaystyle a_1y(a)+a_2D_{\beta^{-1}}y\left(a\right)=0 \\ [0.5em]
\displaystyle b_1y(b)+b_2D_{\beta^{-1}}y(b)=0\,.
\end{array}\right.
\end{equation}
We also assume that $\,r(t)\,$ is a real valued continuous function on $\,[a,b]\,$
and $\,\vert a_1\vert+\vert a_2\vert\neq 0\neq \vert b_1\vert+\vert b_2\vert$.

\noindent This operator $\,\ell_{\beta}^{a,b}\,$ defined by
(\ref{ell-beta-Operator-a-b})-(\ref{Boundary-Conditions-a-b}) satisfies the
following Corollaries.
\begin{cor}
If $\,y\,,\,z\,\in L_{\beta}^2\left(a,b\right)\,$ then
\begin{equation*}
\int_{a}^{b}\left[\Big(\ell_{\beta}^{a,b}y(t)\Big)\overline{z(t)}-
y(t)\overline{\Big(\ell_{\beta}^{a,b}z(t)\Big)}\right]d_{\beta}u=
[y,\overline{z}](b)-[y,\overline{z}]\left(a\right)\,,
\end{equation*}
with $\,[y,z]\,$ defined by (\ref{beta-1-Wronskian}).
\end{cor}
\begin{proof}
To prove this Corollary one follows the procedure of the proof of Theorem
\ref{beta-Lagrange-Identity-Theorem} and make use of Corollary \ref{beta-Lemma3.2-a-b}.
\end{proof}
\begin{cor}
The $\,\beta$-Sturm Liouville eigenvalue problem
(\ref{ell-beta-Operator-a-b}-\ref{Boundary-Conditions-a-b}) is self-adjoint in
$L_{\beta}^2\left(a,b\right)$
\big(i.e., $\ell_{\beta}^{a,b}$ is self-adj. in
$\big\{y\in L_{\beta}^2\left(a,b\right)\!:\,y
\mbox{\,satisfies\,(\ref{Boundary-Conditions-a-b})}\big\}$\big).
\end{cor}

\begin{cor}
All eigenvalues of the problem (\ref{ell-beta-Operator-a-b})-(\ref{Boundary-Conditions-a-b})
are real. Eigenfunctions corresponding to different eigenvalues are orthogonal.
\end{cor}

\subsection{A particular case}\label{particular-case}
The generality of the $\,\beta$-SLP (\ref{ell-beta-Operator})-(\ref{Boundary-Conditions})
or (\ref{ell-beta-Operator-a-b})-(\ref{Boundary-Conditions-a-b})
implies the difficulty on finding explicit non-trivial solutions. The aim of this subsection is
to consider some particular cases, although very comprehensive, in such a way that enables one
to exhibit solutions of equation (\ref{ell-beta-Operator}) or (\ref{ell-beta-Operator-a-b}).

In the following, $\,I\subseteq\mathbb{R}\,$ is an interval and the function
$\,\beta:\,I\rightarrow\,I\,$ satisfies the conditions introduced in subsection
\ref{Subsect-2.1}. Furthermore, the function $\,p\,$ satisfies the same assumptions
assumed in the beginning of subsection \ref{beta-exponential-trigonometric}.
We have the following Proposition relative to the $\,\beta$-exponential and
$\,\beta$-trigonometric functions (\ref{beta-exp})-(\ref{beta-Trigonometric}).
\begin{prop}
The following identities hold on $\,I\,$:
\begin{equation*}
\begin{array}{l}
D_{\beta^{-1}}\,\mbox{{\upshape e}}_{p,\beta}\,(t)=p\big(\beta^{-1}(t)\big)\,\mbox{\upshape e}_{p,\beta}\,(t)\,,\;\:
D_{\beta^{-1}}\,\mbox{\upshape E}_{p,\beta}\,(t)=p\big(\beta^{-1}(t)\big)\,\mbox{\upshape E}_{p,\beta}\big(\beta(t)\big)\,, \\ [0.8em]
D_{\beta^{-1}}\,\mbox{\upshape sin}_{p,\beta}\,(t)=p\big(\beta^{-1}(t)\big)\,\mbox{\upshape cos}_{p,\beta}\,(t)\,,\;\:
D_{\beta^{-1}}\,\mbox{\upshape cos}_{p,\beta}\,(t)=-p\big(\beta^{-1}(t)\big)\,\mbox{\upshape sin}_{p,\beta}\,(t)\,, \\ [0.8em]
D_{\beta^{-1}}\,\mbox{\upshape Sin}_{p,\beta}\,(t)=p\big(\beta^{-1}(t)\big)\,\mbox{\upshape Cos}_{p,\beta}\,(t)\,, \\ [0.8em]
D_{\beta^{-1}}\,\mbox{\upshape Cos}_{p,\beta}\,(t)=-p\big(\beta^{-1}(t)\big)\,\mbox{\upshape Sin}_{p,\beta}\,(t)\,.
\end{array}
\end{equation*}
\end{prop}
\begin{proof}
We prove the first identity.
\begin{equation*}
\begin{array}{l}
D_{\beta^{-1}}\,\mbox{e}_{p,\beta}(t)\,=\,\displaystyle
\frac{\mbox{e}_{p,\beta}(t)-\mbox{e}_{p,\beta}\big(\beta^{-1}(t)\big)}{t-\beta^{-1}(t)}= \\ [1.8em]
=\displaystyle\frac{1-\frac{1}{1-p\left(\beta^{-1}(t)\right)\Big(\beta^{-1}(t)-t\Big)}}{
\big(t-\beta^{-1}(t)\big)\displaystyle\prod_{k=0}^{\infty}
\left[1-p\left(\beta^k(t)\right)\Big(\beta^k(t)-\beta^{k+1}(t)\Big)\right]} \\ [3em]
=\displaystyle-\frac{p\left(\beta^{-1}(t)\right)\Big(\beta^{-1}(t)-t\Big)}{
t-\beta^{-1}(t)}\,\mbox{e}_{p,\beta}(t)=p\left(\beta^{-1}(t)\right)\mbox{e}_{p,\beta}(t)\,.
\end{array}
\end{equation*}
\end{proof}
In a similar way to the $\,\beta$-product rule (\ref{beta-product-rule}) we have
\begin{equation}\label{beta-1-product-rule}
\begin{array}{lll}
D_{\beta^{-1}}[f\cdot g](t) & = &
D_{\beta^{-1}}[f](t)\cdot g(t)+f\big(\beta^{-1}(t)\big)\cdot D_{\beta^{-1}}[g](t) \\ [0.8em]
 & = & D_{\beta^{-1}}[g](t)\cdot f(t)+g\big(\beta^{-1}(t)\big)\cdot D_{\beta^{-1}}[f](t)\,.
 \end{array}
\end{equation}
Using the corresponding definitions and this latter formula we obtain the following properties.
\begin{prop}
The following identities hold on $\,I\,$:
\begin{equation*}
\begin{array}{l}
D_{\beta^{-1}}D_{\beta}\,\mbox{\upshape e}_{p,\beta}(t)=\left[p^2\big(\beta^{-1}(t)\big)+
D_{\beta^{-1}}p(t)\right]\mbox{\upshape e}_{p,\beta}(t)\,, \\ [0.8em]
D_{\beta^{-1}}D_{\beta}\,\mbox{\upshape E}_{p,\beta}(t)=
p(t)p\big(\beta^{-1}(t)\big)\mbox{\upshape E}_{p,\beta}(t)+D_{\beta^{-1}}p(t)\,\mbox{\upshape E}_{p,\beta}\big(\beta(t)\big)\,, \\ [0.8em]
D_{\beta^{-1}}D_{\beta}\,\mbox{\upshape sin}_{p,\beta}(t)=
-p^2\big(\beta^{-1}(t)\big)\,\mbox{\upshape sin}_{p,\beta}(t)+
D_{\beta^{-1}}p(t)\,\mbox{\upshape cos}_{p,\beta}(t)\,, \\ [0.8em]
D_{\beta^{-1}}D_{\beta}\,\mbox{\upshape cos}_{p,\beta}(t)=
-p^2\big(\beta^{-1}(t)\big)\,\mbox{\upshape cos}_{p,\beta}(t)-
D_{\beta^{-1}}p(t)\,\mbox{\upshape sin}_{p,\beta}(t)\,, \\ [0.8em]
D_{\beta^{-1}}D_{\beta}\,\mbox{\upshape Sin}_{p,\beta}(t)=
-p^2\big(\beta^{-1}(t)\big)\,\mbox{\upshape Sin}_{p,\beta}(t)+
D_{\beta^{-1}}p(t)\,\mbox{\upshape Cos}_{p,\beta}(t)\,, \\ [0.8em]
D_{\beta^{-1}}D_{\beta}\,\mbox{\upshape Cos}_{p,\beta}(t)=
-p^2\big(\beta^{-1}(t)\big)\,\mbox{\upshape Cos}_{p,\beta}(t)-
D_{\beta^{-1}}p(t)\,\mbox{\upshape Sin}_{p,\beta}(t)\,.
\end{array}
\end{equation*}
\end{prop}
\begin{proof}
The proofs are straightforward.
Just to illustrate it we give the proof of the second identity.

\noindent By (\ref{trigonometric-beta-derivatives}) and (\ref{beta-1-product-rule}) we obtain
\begin{equation*}
\begin{array}{lll}
D_{\beta^{-1}}D_{\beta}\mbox{E}_{p,\beta}(t)\,& = &\,
D_{\beta^{-1}}\Big[p(t)\mbox{E}_{p,\beta}\big(\beta(t)\big)\Big] \\ [0.8em]
 & = & D_{\beta^{-1}}p(t)\mbox{E}_{p,\beta}\big(\beta(t)\big)+
 p\big(\beta^{-1}(t)\big)D_{\beta^{-1}}\mbox{E}_{p,\beta}\big(\beta(t)\big)\,,
 \end{array}
\end{equation*}
hence, by (i) of Lemma \ref{beta-Lemma3.2},
\begin{equation*}
\begin{array}{lll}
D_{\beta^{-1}}D_{\beta}\mbox{E}_{p,\beta}(t) & = &
D_{\beta^{-1}}p(t)\mbox{E}_{p,\beta}\big(\beta(t)\big)+
 p\big(\beta^{-1}(t)\big)D_{\beta}\mbox{E}_{p,\beta}(t) \\ [1em]
& = & p(t)p\big(\beta^{-1}(t)\big)\mbox{E}_{p,\beta}(t)+
D_{\beta^{-1}}p(t)\mbox{E}_{p,\beta}\big(\beta(t)\big)\,.
 \end{array}
\end{equation*}
\end{proof}
For the particular case where $\,p\,$ is the constant function $\,p(t)=z\,$,
the following Corollaries hold.
\begin{cor}
If $\,p\,$ is the constant function $\,p(t)=z\,$ on $\,I\,$ then the following identities
hold:
\begin{equation*}
\begin{array}{l}
D_{\beta^{-1}}D_{\beta}\,\mbox{\upshape e}_{z,\beta}(t)=z^2\,\mbox{\upshape e}_{z,\beta}(t) \,,\quad
D_{\beta^{-1}}D_{\beta}\,\mbox{\upshape E}_{z,\beta}(t)=z^2\,\mbox{\upshape E}_{z,\beta}(t) \,, \\ [0.8em]
D_{\beta^{-1}}D_{\beta}\,\mbox{\upshape sin}_{z,\beta}(t)=-z^2\,\mbox{\upshape sin}_{z,\beta}(t) \,,\quad
D_{\beta^{-1}}D_{\beta}\,\mbox{\upshape cos}_{z,\beta}(t)=-z^2\,\mbox{\upshape cos}_{z,\beta}(t) \,, \\ [0.8em]
D_{\beta^{-1}}D_{\beta}\,\mbox{\upshape Sin}_{z,\beta}(t)=-z^2\,\mbox{\upshape Sin}_{z,\beta}(t) \,,\quad
D_{\beta^{-1}}D_{\beta}\,\mbox{\upshape Cos}_{z,\beta}(t)=-z^2\,\mbox{\upshape Cos}_{z,\beta}(t)\,.
\end{array}
\end{equation*}
\end{cor}
\begin{cor}\label{beta-beta-1}
Let $\,p\,$ be the constant function $\,p(t)=z\,$ on $\,I\,$.
If the function $\,\beta\,$ satisfies $\,D_{\beta}\beta^{-1}(t)=k\,$ where $\,k\,$ is
independent of $\,t\,$ then, the $\,\beta$-exponentials functions
(\ref{beta-exp})-(\ref{beta-Exp}) satisfy (\ref{ell-beta-Operator})
or (\ref{ell-beta-Operator-a-b}) with $\,r(t)=0\,$ in $\,I\,$ and $\,\lambda=kz^2\,$
while each of the $\,\beta$-trigonometric functions
(\ref{beta-trigonometric})-(\ref{beta-Trigonometric}) satisfy (\ref{ell-beta-Operator})
or (\ref{ell-beta-Operator-a-b}) with $\,r(t)=0\,$ in $\,I\,$ and $\,\lambda=-kz^2\,$.
\end{cor}
\begin{rem}
\begin{enumerate}
\item If $\,\beta(t)=qt+\omega\,$, which corresponds to the $(q,\omega)$-derivative
operator (\ref{Hahn-operator}), the condition on the function $\,\beta\,$ of Corollary
\ref{beta-beta-1} is satisfied since $\,D_{\beta}\beta^{-1}(t)=1/q\,$ in $\,I\,$.
\item Notice that, under the conditions of Corollary \ref{beta-beta-1}, the
$\,\beta$-exponentials and $\,\beta$-trigonometric functions, for appropriate choices of
the constants $\,a_1, a_2, b_1, b_2\,$, and $\,a, b\,$ are solutions of both the
$\,\beta$-Sturm Liouville problems (\ref{ell-beta-Operator})-(\ref{Boundary-Conditions}) or
(\ref{ell-beta-Operator-a-b})-(\ref{Boundary-Conditions-a-b}).
\item An interesting feature of this setting would be to obtain an explicit function $\,\beta\,$
under the assumptions of Corollary \ref{beta-beta-1}, but not coincident with the one for the
Hahn operator (\ref{Hahn-operator}), and specific values of $\,a_1, a_2, b_1, b_2\,$, and
$\,a, b\,$ in order to obtain a fundamental set of soluttions of the corresponding
$\,\beta$-SLP. This could pave the way to pursue investigations in other directions.
\end{enumerate}
\end{rem}

\vspace{1em}
\textbf{Final conclusions}.
We established conditions that make possible a self-adjoint formulation of the Sturm-Liouville
problem (\ref{ell-beta-Operator})-(\ref{Boundary-Conditions}) in the space
$\,L_{\beta}^2\left(s_0,b\right)$ and we were able to extend it to the problem
(\ref{ell-beta-Operator-a-b})-(\ref{Boundary-Conditions-a-b}) in
$\,L_{\beta}^2\left(a,b\right)$.
Its construction is based in the general $\beta$-difference quantum operator
(\ref{beta-derivative}).
Then, it was proved that all the corresponding eigenvalues are real and the relative
eigenfunctions are orthogonal.
We believe that the establishment of this frame will allow other researchers to get interested
in it and push further to other directions and results.


\vspace{1em}
\begin{thebibliography}{1}\UseRawInputEncoding

\bibitem{AA-NC:2019} L. D. Abreu, R. \'Alvarez-Nodarse, J. L. Cardoso,
\textit{Uniform convergence of basic Fourier-Bessel series on a $q$-linear grid},
Ramanujan J., \textbf{49} (2019), 421--449.

\noindent https://doi.org/10.1007/s11139-018-0070-3


\bibitem{A:2004} L. D. Abreu,
\textit{A $\,q$-sampling theorem related to the $\,q$-Hankel transform},
Proc. Amer. Math. Soc., \textbf{133}(4) (2004), 1197--1203.


\bibitem{A:2006} L.D. Abreu,
\textit{Functions $q$-orthogonal with respect to their own zeros},
Proc. Amer. Math. Soc., Volume \textbf{134}, Number 9 (2006), 2695--2701.


\bibitem{A:2007} L.D. Abreu,
\textit{Real Paley-Wiener theorems for the Koornwinder-Swarttouw
$q$-Hankel transform}, J. Math. Anal. Appl. \textbf{334} (2007), 223--231.


\bibitem{ABC:2003} L. D. Abreu, J. Bustoz, and J. L. Cardoso,
\textit{The roots of The Third Jackson $q$-Bessel Function},
Intern. J. Math. Math. Sci.,
Volume 2003, \textbf{67} (2003), 4241--4248.


\bibitem{AAIM:2007} M. H. Abu-Risha, M. H. Annaby, M. E. H. Ismail, Z. S. Mansour,
\textit{Linear $q$-difference equations},
Z. Anal. Anwend. \textbf{26} (2007), 481--494.


\bibitem{A:2009} K. A. Aldwoah,
\textit{Generalized time scales and associated difference equation},
PhD. thesis, Cairo University, 2009.


\bibitem{AMT:2012} K. A. Aldwoah, A. B. Malinowska and D. F. M. Torres,
\textit{The power quantum calculus and variational problems},
Dyn. Contin. Discrete Impuls. Syst. Ser. B Appl. Algorithms,
Vol 19 (2012), no. 1-2, 93--116.


\bibitem{A:2003} M. H. Annaby,
\textit{$\,q$-type sampling theorems},
Result. Math. \textbf{44} (2003), 214--225.


\bibitem{ABI:2007} M. H. Annaby, J. Bustoz and M. E. H. Ismail,
\textit{On sampling theory and basic Sturm-Liouville systems},
J. Comput. Appl. Math. \textbf{206} (2007), 73--85.


\bibitem{AH:2018} M. H. Annaby and H. A. Hassan,
\textit{Sampling theorems for Jackson-N\"orlund transforms associated with
Hahn-difference operators},
J. Math. Anal. Appl. \textbf{464} (2018), 493--506.


\bibitem{AHA:2012} M. H. Annaby, A. E. Hamza and K. A. Aldwoah,
\textit{Hahn Difference Operator and Associated Jackson-N\"orlund Integrals},
J. Optim. Theory Appl. \textbf{154} (2012), no.1, 133--153.


\bibitem{AHM:2012} M. H. Annaby, H. A. Hassan and Z. S. Mansour,
\textit{Sampling theorems associated with singular $\,q$-Sturm Liouville problems},
Results. Math. \textbf{62} (2012), 121--136.


\bibitem{AHM:2018} M. H. Annaby, A. E. Hamza, and S. D. Makharesh,
\textit{Chapter 4 - A Sturm-Liouville Theory for Hahn Difference Operator},
Frontiers in Orthogonal Polynomials and $q$-Series
(Edited by: M. Zuhair Nashed and Xin Li) 2018, pp. 35--83,
from Contemporary Mathematics and Its Applications: Monographs, Expositions and Lecture
Notes: Volume 1.

\noindent https://doi.org/10.1142/9789813228887\textunderscore 0004


\bibitem{AM:2005} M. H. Annaby and Z. S. Mansour,
\textit{Basic Sturm-Liouville problems},
J. Phys. A: Math Gen. \textbf{38} (2005), 3775--3797.


\bibitem{AM:2012} M. H. Annaby and Z. S. Mansour,
\textit{$q$-Fractional Calculus and Equations},
Springer, Berlin, 2012.


\bibitem{BC:2001} J. Bustoz and J. L. Cardoso,
\textit{Basic Analog of Fourier Series on a $q$-Linear Grid},
J. Approx. Theory {\bf 112} (2001), 134--157.


\bibitem{C:2006} J. L. Cardoso,
\textit{Basic Fourier series on a $\,q$-linear grid: convergence
theorems}, J. Math. Anal. Appl. \textbf{323} (2006), 313--330.


\bibitem{C:2011} J. L. Cardoso,
\textit{Basic Fourier series: convergence on and outside the $q-$linear grid},
J. Fourier Anal. Appl. \textbf{17}(1) (2011), 96--114.


\bibitem{C:2016} J. L. Cardoso,
\textit{A few properties of the Third Jackson
$q$-Bessel Function}, Analysis Mathematica \textbf{42:4} (2016), 323--337.


\bibitem{C:2018} J. L. Cardoso,
\textit{On basic Fourier-Bessel expansions},
SIGMA \textbf{14}, 035 (2018), 13 pages 


\bibitem{C:2020}
J. L. Cardoso,
\textit{Variations around a general quantum operator},
Ramanujan J. (2020). https://doi.org/10.1007/s11139-019-00210-8


\bibitem{CP:2015}
J. L. Cardoso and J. Petronilho,
\textit{Variations around Jackson's quantum operator},
Methods Appl. Anal. {\bf 22} (2015), no. 4, 343--358.


\bibitem{CM:2018} A. M. Cruz and N. Martins,
\textit{General quantum variational calculus},
Stat. Optim. Inf. Comput. \textbf{6} (2018), 22--41.


\bibitem{CMT:2012} A. M. C. Brito da Cruz, N. Martins, and D. F. M. Torres,
\textit{Higher-order Hahn's quantum variational calculus},
Nonlinear Anal. \textbf{75} (2012), no. 3, 1147--1157.


\bibitem{DBF:2009} L. Dhaouadi, W. Binous and A. Fitouhi,
\textit{Paley–Wiener theorem for the $q$-Bessel transform and associated $q$-sampling formula},
Expo. Math. \textbf{27} (2009), 55-–72


\bibitem{FD:2004} A. Fitouhi and L. Dhaouadi,
\textit{On a $q$-Paley-Wiener theorem},
J. Math. Anal. Appl. \textbf{294} (2004), 17-–23


\bibitem{E:1983} H. Exton,
\textit{$q$-Hypergeometric Functions and Applications},
Halsted Press, New York, 1993.


\bibitem{E:1992} H. Exton,
\textit{Basic Sturm-Liouville theory},
Rev. Tecn. Fac. Ing., Univ. Zulia \textbf{11} (1992), 85--100.


\bibitem{FSZ:2017} N. Faried, E. M. Shehata and R. M. Zafarani,
\textit{On homogeneous second order linear general quantum difference equations},
J. Inequal. Appl. (2017) 2017:198, 1--13.
DOI 10.1186/s13660-017-1471-3


\bibitem{FSZ:2018} N. Faried, E. M. Shehata and R. M. Zafarani,
\textit{Theory of $n$th-order linear general quantum difference equations},
Adv. Difference Equ. (2018) 2018:264, 1--24.

\noindent https://doi.org/10.1186/s13662-018-1715-7


\bibitem{H:2019} H. A. Hassan,
\textit{A Completeness Theorem for a Hahn-Fourier System and an Associated Classical Sampling Theorem},
Results Math \textbf{74:34} (2019).


\bibitem{HSSA:2015} A. E. Hamza, A. M. Sarhan, E. M. Shehata and K. A. Aldwoah,
\textit{A general quantum difference calculus},
Adv. Difference Equ. \textbf{2015:182} (2015), 1--19.


\bibitem{HS:2015} A. E. Hamza and E. M. Shehata,
\textit{Some inequalities based on a general quantum difference operator},
J. Inequal. Appl. \textbf{2015:38} (2015),  1--12.


\bibitem{HS:2016} A. E. Hamza and E. M. Shehata,
\textit{Existence and Uniqueness of Solutions of General Quantum Difference Equations},
Adv.Dyn.Syst.Appl. \textbf{11} (2016), N. 1, 45--58.


\bibitem{HSS:2017} A. E. Hamza, A. M. Sarhan and E. M. Shehata,
\textit{Exponential, Trigonometric and Hyperbolic Functions Associated with a
General Quantum Difference Operator},
Adv. Dyn. Syst. Appl. \textbf{12} (2017), Number 1, 25--38.


\bibitem{IZ:2003} M. E. H. Ismail and A. I. Zayed,
\textit{A $\,q$-analogue of the Whittaker-Shannon-Kotel'nikov sampling theorem},
Proc. Amer. Math. Soc. \textbf{183} (2003) 3711--3719.


\bibitem{JCF:2017} L. Jia, J. Cheng and Z. Feng,
\textit{A $q$-analogue of Kummers equation},
Electron. J. Differential Equations, Vol. 2017 (2017), No. 31, 1-–20.


\bibitem{P:2007}J. Petronilho,
\textit{Generic formulas for the values at the singular points of some special monic
classical $H_{q,\omega}$-orthogonal polynomials},
J. Comput. Appl. Math. \textbf{205} (2007), 1, 314--324.


\bibitem{KS:1992} T.H. Koornwinder and R.F. Swarttouw,
\textit{On $q$-analogues of the Fourier and Hankel transforms},
Trans. Amer. Math. Soc. \textbf{333} (1992), 445--461.


\bibitem{M:2016} Z. S. Mansour,
\textit{Variational methods for fractional $\,q$-Sturm--Liouville Problems},
Bound. Value Probl. 2016:150 (2016).


\bibitem{MT:2010} A. B. Malinowska and D. F. M. Torres,
\textit{The Hahn quantum variational calculus},
J. Optim. Theory Appl. 147 (2010), no. 3, 419--442.


\bibitem{MT:2014} A. B. Malinowska and D. F. M. Torres,
\textit{Quantum variational calculus},
Springer Briefs in Electrical and Computer Engineering:
Control, Automation, and Robotics, Springer, 2014.


\bibitem{SFZ:2020} E. M. Shehata, N. Faried and R. M. Zafarani,
\textit{A general quantum Laplace transform},
Adv. Difference Equ. 2020, 613(2020).

\noindent https://doi.org/10.1186/s13662-020-03070-5


\bibitem{S:2003} S. K. Suslov,
\textit{An introduction to basic Fourier series}. With a foreword by
Mizan Rahman. Developments in Mathematics {\bf 9}, {Kluwer
Academic Publishers}, Dordrecht, 2003.


\bibitem{T:1962} E. C. Titchmarsh,
\textit{Eigenfunction Expansions Associated with Second Order Differential Equations},
Part I. 2nd Edition, The Clarendon Press, Oxford, 1962.

\end{thebibliography}
\end{document}


\title{Corrigendum}

\vspace{2em}
\centerline{A $\beta-$Sturm Liouville problem associated with the general quantum operator}

\vspace{0.5em}
\centerline{arXiv 2101.04217 v2 [math.CA] 19 May 2021}

\author[Cardoso]{J.L. Cardoso}

\maketitle

The definitions of $\,\mbox{sin}_{p,\beta}(t)\,$ and
$\,\mbox{Sin}_{p,\beta}(t)\,$ in (2.12) and (2.13), page 6, should be
$$\mbox{sin}_{p,\beta}(t)=\frac{\mbox{e}_{ip,\beta}(t)-\mbox{e}_{-ip,\beta}(t)}{2i}\quad,\qquad
\mbox{Sin}_{p,\beta}(t)=\frac{\mbox{E}_{ip,\beta}(t)-\mbox{E}_{-ip,\beta}(t)}{2i}\,.$$

\vspace{0.3em}
Also at page 6, the last two identities of formulae (2.14) must be replaced by
$$D_{\beta}\,\mbox{Sin}_{p,\beta}(t)=p(t)\,\mbox{Cos}_{p,\beta}(\beta(t))\,,\quad
D_{\beta}\,\mbox{Cos}_{p,\beta}(t)=-p(t)\,\mbox{Sin}_{p,\beta}(\beta(t))\,.$$

\vspace{1em}
With respect to Proposition 3.11, page 12, the first four identities aren't correct.
Its full statement is:

\vspace{0.5em}
\noindent
\textbf{Proposition 3.11.}
\emph{The following identities hold on} $\,I\,$:
\begin{equation*}
\begin{array}{l}
D_{\beta^{-1}}\,\mbox{{\upshape e}}_{p,\beta}\,(t)=p\big(\beta^{-1}(t)\big)\,\mbox{\upshape e}_{p,\beta}\,\big(\beta^{-1}(t)\big)\,,\;\:
D_{\beta^{-1}}\,\mbox{\upshape E}_{p,\beta}\,(t)=p\big(\beta^{-1}(t)\big)\,\mbox{\upshape E}_{p,\beta}(t)\,, \\ [0.8em]
D_{\beta^{-1}}\,\mbox{\upshape sin}_{p,\beta}\,(t)=p\big(\beta^{-1}(t)\big)\,\mbox{\upshape cos}_{p,\beta}\,\big(\beta^{-1}(t)\big)\,,\;\,
D_{\beta^{-1}}\,\mbox{\upshape cos}_{p,\beta}\,(t)=-p\big(\beta^{-1}(t)\big)\,\mbox{\upshape sin}_{p,\beta}\,\big(\beta^{-1}(t)\big)\,, \\ [0.8em]
D_{\beta^{-1}}\,\mbox{\upshape Sin}_{p,\beta}\,(t)=p\big(\beta^{-1}(t)\big)\,\mbox{\upshape Cos}_{p,\beta}\,(t)\,, \\ [0.8em]
D_{\beta^{-1}}\,\mbox{\upshape Cos}_{p,\beta}\,(t)=-p\big(\beta^{-1}(t)\big)\,\mbox{\upshape Sin}_{p,\beta}\,(t)\,.
\end{array}
\end{equation*}

\vspace{0.5em}
Also regarding Proposition 3.11, in the last line of the first identity's proof,
one should read:
$$
\,=\,-\frac{p\left(\beta^{-1}(t)\right)\Big(\beta^{-1}(t)-t\Big)}{
t-\beta^{-1}(t)}\,\mbox{e}_{p,\beta}(\beta^{-1}(t))=
p\left(\beta^{-1}(t)\right)\mbox{e}_{p,\beta}(\beta^{-1}(t))\,.
$$

\vspace{1em}
Along with this, one finds a new statement of Proposition 3.12 at page 12.

\vspace{0.5em}
\noindent
\textbf{Proposition 3.12.}
The following identities hold on $\,I\,$:
\begin{equation*}
\begin{array}{l}
D_{\beta^{-1}}D_{\beta}\,\mbox{\upshape e}_{p,\beta}(t)=\left[\,p(t)p\big(\beta^{-1}(t)\big)+
D_{\beta^{-1}}p(t)\right]\mbox{\upshape e}_{p,\beta}(\beta^{-1}(t))\,, \\ [0.8em]
D_{\beta^{-1}}D_{\beta}\,\mbox{\upshape E}_{p,\beta}(t)=
\left[\,p(t)p\big(\beta^{-1}(t)\big)D_{\beta^{-1}}\beta(t)+D_{\beta^{-1}}p(t)\,\right]
\mbox{\upshape E}_{p,\beta}\big(\beta(t)\big)\,, \\ [0.8em]
D_{\beta^{-1}}D_{\beta}\,\mbox{\upshape sin}_{p,\beta}(t)=
-p^2\big(\beta^{-1}(t)\big)\,\mbox{\upshape sin}_{p,\beta}\big(\beta^{-1}(t)\big)+
D_{\beta^{-1}}p(t)\,\mbox{\upshape cos}_{p,\beta}(t)\,, \\ [0.8em]
D_{\beta^{-1}}D_{\beta}\,\mbox{\upshape cos}_{p,\beta}(t)=
-p^2\big(\beta^{-1}(t)\big)\,\mbox{\upshape cos}_{p,\beta}\big(\beta^{-1}(t)\big)-
D_{\beta^{-1}}p(t)\,\mbox{\upshape sin}_{p,\beta}(t)\,, \\
[0.8em]D_{\beta^{-1}}D_{\beta}\,\mbox{\upshape Sin}_{p,\beta}(t)=
-p^2(t)D_{\beta^{-1}}\beta(t)\,\mbox{\upshape Sin}_{p,\beta}\big(\beta(t)\big)+
D_{\beta^{-1}}p(t)\,\mbox{\upshape Cos}_{p,\beta}(t)\,, \\ [0.8em]
D_{\beta^{-1}}D_{\beta}\,\mbox{\upshape Cos}_{p,\beta}(t)=
-p^2(t)D_{\beta^{-1}}\beta(t)\,\mbox{\upshape Cos}_{p,\beta}\big(\beta(t)\big)-
D_{\beta^{-1}}p(t)\,\mbox{\upshape Sin}_{p,\beta}(t)\,.
\end{array}
\end{equation*}

\vspace{0.5em}
Regarding the proof of the second identity of Proposition 3.12, its last three lines
must be replaced by:

\noindent
hence, since
$$D_{\beta^{-1}}\mbox{E}_{p,\beta}\big(\beta(t)\big)=
\frac{\mbox{E}_{p,\beta}\big(\beta(t)\big)-\mbox{E}_{p,\beta}(t)}{t-\beta^{-1}(t)}=
\frac{\mbox{E}_{p,\beta}(t)-\mbox{E}_{p,\beta}\big(\beta(t)\big)}{t-\beta(t)}
\frac{\beta(t)-t}{t-\beta^{-1}(t)}=D_{\beta}\mbox{\upshape E}_{p,\beta}(t)
D_{\beta^{-1}}\beta(t)\,,$$
then,
\begin{equation*}
\begin{array}{lll}
D_{\beta^{-1}}D_{\beta}\mbox{E}_{p,\beta}(t) & = &
D_{\beta^{-1}}p(t)\mbox{E}_{p,\beta}\big(\beta(t)\big)+
 p\big(\beta^{-1}(t)\big)D_{\beta^{-1}}\beta(t)D_{\beta}\mbox{\upshape E}_{p,\beta}(t) \\ [1em]
& = & p(t)p\big(\beta^{-1}(t)D_{\beta^{-1}}\beta(t)\big)\mbox{\upshape E}_{p,\beta}\big(\beta(t)\big)+
D_{\beta^{-1}}p(t)\mbox{E}_{p,\beta}\big(\beta(t)\big)\,.
 \end{array}
\end{equation*}

\vspace{1em}
Following the previous amendments comes Corollary 3.13.

\vspace{0.5em}
\noindent
\textbf{Corollary 3.13.}
\emph{If} $\,p\,$ \emph{is the constant function} $\,p(t)=z\,$ \emph{on} $\,I\,$
\emph{then the following identities hold:}

\begin{equation*}
\begin{array}{l}
D_{\beta^{-1}}D_{\beta}\,\mbox{\upshape e}_{z,\beta}(t)=z^2\,\mbox{\upshape e}_{z,\beta}\big(\beta^{-1}(t)\big) \,,\quad
D_{\beta^{-1}}D_{\beta}\,\mbox{\upshape E}_{z,\beta}(t)=z^2D_{\beta^{-1}}\beta(t)\,\mbox{\upshape E}_{z,\beta}\big(\beta(t)\big) \,, \\ [0.8em]
D_{\beta^{-1}}D_{\beta}\,\mbox{\upshape sin}_{z,\beta}(t)=-z^2\,\mbox{\upshape sin}_{z,\beta}\big(\beta^{-1}(t)\big) \,,\quad
D_{\beta^{-1}}D_{\beta}\,\mbox{\upshape cos}_{z,\beta}(t)=-z^2\,\mbox{\upshape cos}_{z,\beta}\big(\beta^{-1}(t)\big) \,, \\ [0.8em]
D_{\beta^{-1}}D_{\beta}\,\mbox{\upshape Sin}_{z,\beta}(t)=-z^2D_{\beta^{-1}}\beta(t)\,\mbox{\upshape Sin}_{z,\beta}\big(\beta(t)\big) \,,\quad
D_{\beta^{-1}}D_{\beta}\,\mbox{\upshape Cos}_{z,\beta}(t)=-z^2D_{\beta^{-1}}\beta(t)\,\mbox{\upshape Cos}_{z,\beta}\big(\beta(t)\big) \,.
\end{array}
\end{equation*}

\vspace{1em}
\vspace{0.5em}
Remarking that $\,1+z\big(t-\beta(t)\big)\neq 0\,$ holds on $\,I\,$ whenever $\,z\,$ is
a constant on $\,I\,$ then, adjusting Corollary 3.14 to Corollary 3.13, we obtain:

\vspace{0.5em}
\noindent
\textbf{Corollary 3.14.}
Let $\,p\,$ be the constant function $\,p(t)=z\,$ on $\,I\,$.
Then, the $\,\beta$-exponential function ($2.11$) satisfy equation
$$-D_{\beta}\beta^{-1}(t)D_{\beta^{-1}}D_{\beta}\,\,\mbox{\upshape E}_{z,\beta}(t)=
-z^2\,\,\mbox{\upshape E}_{z,\beta}\big(\beta(t)\big)\,,$$ or, equivalently, satisfy
$$-D_{\beta}\beta^{-1}(t)D_{\beta^{-1}}D_{\beta}\,\mbox{\upshape E}_{z,\beta}(t)=
-\frac{z^2}{1+z\big(t-\beta(t)\big)}\,\mbox{\upshape E}_{z,\beta}(t)\,,$$
which shows that $\,\mbox{\upshape E}_{z,\beta}(t)\,$ satisfy ($3.3$)
or ($3.10$) with $\,r(t)=\frac{z^2}{1+z\big(t-\beta(t)\big)}\,$ in $\,I\,$ and $\,\lambda=0\,$.

\vspace{1em}
As a consequence of this new wording, the item (2) of \emph{Remark} 3.15 is no
longer valid.